\documentclass[letterpaper,journal,onecolumn,12pt]{IEEEtran}

\renewcommand{\footnoterule}{%
  \kern-3pt
  \hrule width 0.4\columnwidth
  \kern 2.6pt
}

\usepackage[utf8]{inputenc} 
\usepackage[T1]{fontenc}
\usepackage{verbatim}

\usepackage[backend=bibtex, style = numeric, doi = false, url = false, isbn = false, maxbibnames = 6]{biblatex}
\bibliography{references.bib}
\renewbibmacro{in:}{}      
\newbibmacro{string+doi}[1]{\iffieldundef{doi}{#1}{\href{https://dx.doi.org/\thefield{doi}}{#1}}}
\DeclareFieldFormat{title}{\usebibmacro{string+doi}{\mkbibemph{#1}}}
\DeclareFieldFormat[article]{title}{\usebibmacro{string+doi}{\mkbibquote{#1}}}
\DeclareFieldFormat[incollection]{title}{\usebibmacro{string+doi}{\mkbibquote{#1}}}                   
\DeclareFieldFormat[inproceedings]{title}{\usebibmacro{string+doi}{\mkbibquote{#1}}}     

\usepackage[colorlinks=true,linkcolor=blue,urlcolor=blue,citecolor=blue,anchorcolor=green,pdfusetitle]{hyperref}

\usepackage[cmex10]{amsmath}  
\usepackage{amsfonts}
\usepackage{amssymb,amsthm,mathtools}

\usepackage{cleveref}

\usepackage{mleftright}       
\mleftright                   

\usepackage{graphicx}         
\usepackage{booktabs}         

\usepackage{aliascnt}

\newtheorem{thm}{Theorem}

\newaliascnt{prop}{thm}
\newtheorem{prop}[prop]{Proposition}
\aliascntresetthe{prop}

\newaliascnt{lem}{thm}
\newtheorem{lem}[lem]{Lemma}
\aliascntresetthe{lem}

\newaliascnt{cor}{thm}
\newtheorem{cor}[cor]{Corollary}
\aliascntresetthe{cor}

\theoremstyle{definition}

\newaliascnt{rem}{thm}
\newtheorem{rem}[rem]{Remark}
\aliascntresetthe{rem}

\newaliascnt{defn}{thm}
\newtheorem{defn}[defn]{Definition}
\aliascntresetthe{defn}

\newaliascnt{conj}{thm}

\aliascntresetthe{conj}

\newaliascnt{ex}{thm}

\aliascntresetthe{ex}

\usepackage{cleveref}

\crefname{thm}{theorem}{theorems}
\Crefname{thm}{Theorem}{Theorems}

\crefname{prop}{proposition}{propositions}
\Crefname{prop}{Proposition}{Propositions}

\crefname{lem}{lemma}{lemmas}
\Crefname{lem}{Lemma}{Lemmas}

\crefname{cor}{corollary}{corollaries}
\Crefname{cor}{Corollary}{Corollaries}

\crefname{rem}{remark}{remarks}
\Crefname{rem}{Remark}{Remarks}

\crefname{defn}{definition}{definitions}
\Crefname{defn}{Definition}{Definitions}

\crefname{conj}{conjecture}{conjectures}
\Crefname{conj}{Conjecture}{Conjectures}

\crefname{ex}{example}{examples}
\Crefname{ex}{Example}{Examples}

\usepackage{braket}
\usepackage{mathtools}
\usepackage{dsfont}
\usepackage{xcolor}

\newcommand{\cE}{\mathcal{E}}

\newcommand{\bC}{\mathbb{C}}
\newcommand{\bR}{\mathbb{R}}
\newcommand{\bP}{\mathbb{P}}
\newcommand{\bM}{\mathbb{M}}
\newcommand{\mc}{\mathcal}
\newcommand{\mbb}{\mathbb}

\DeclareMathOperator{\bc}{bc}
\DeclareMathOperator{\spec}{Spec}
\DeclareMathOperator{\GL}{GL}
\DeclareMathOperator{\Herm}{Herm}
\DeclareMathOperator{\tr}{tr}
\DeclareMathOperator{\id}{id}
\DeclareMathOperator{\im}{im} 

\DeclareMathOperator{\Span}{Span}

\DeclareMathOperator{\Gr}{Gr}

\DeclareMathOperator*{\argmin}{arg\,min}
\newcommand{\one}{\mbb{I}}

\definecolor{cool_green}{rgb}{0.0, 0.5, 0.0}

\usepackage[affil-it]{authblk}

\begin{document}
\title{From Majorization to Horn and Unitary-Orbit Inequalities}
\author{Juntai Zhou%
\thanks{\href{mailto:juntaiz2@illinois.edu}{juntaiz2@illinois.edu}}}
\affil{\small Department of Mathematics, University of Illinois Urbana-Champaign}
\date{}
\maketitle

\begin{abstract}
    Majorization inequalities encode spectral comparisons through leading partial sums, whereas complete Horn inequalities capture all admissible selected-eigenvalue constraints associated with sums of Hermitian matrices [Fulton, 2000]. We develop a general framework for lifting majorization inequalities to this complete Horn level, which is equivalent to unitary-orbit inequalities. The main ingredients are a generalized Hersch--Zwahlen variational formula for coordinatewise monotone functions of selected eigenvalues and a Schubert-geometric lifting principle that reduces the desired Horn inequalities to the concavity or convexity of suitable scalarizer functions on matrix compressions. As examples we apply the framework to three families of operator inequalities. First, we obtain a complete Horn extension of the weak-majorization inequality for convex functions established by Aujla and Silva [Aujla and Silva, 2003]. Second, we lift the concavity and convexity results for two-variable positive-power functions due to Carlen, Frank, and Lieb and Zhang [Carlen, Frank, and Lieb, 2016; Zhang, 2020]. Third, we establish multiplicative Horn inequalities for multivariable geodesic means, including the weighted Karcher mean and the two-variable geometric mean, extending determinant-level inequalities of Bourin and Hiai [Bourin and Hiai, 2014]. The resulting statements yield unitary-orbit inequalities that generalize familiar weak-majorization, trace, and determinant inequalities.
\end{abstract}

\section{Introduction}
Many central questions in quantum information theory are spectral:
the eigenvalues of a density operator determine its mixedness and
entropies~\cite{Watrous2018,Wehrl1978},
Schmidt coefficients determine bipartite pure-state
entanglement~\cite{Vidal2000},
and trace functionals of positive operators quantify state
distinguishability~\cite{AudenaertEtAl2007,FuchsCaves1995,ZhouEtAl2025GU}. Majorization gives these data an operational order. If $\lambda(\rho)\prec\lambda(\sigma)$, then $\rho$ is spectrally more mixed than $\sigma$ and every Schur-concave entropy is at least as large; Nielsen's theorem likewise characterizes deterministic bipartite pure-state transformations under local operations and classical communication by majorization of the Schmidt probability vectors \cite{MarshallOlkinArnold2011, Nielsen1999}. Mathematically, majorization compares decreasing partial sums, weak majorization drops equality of the total sums, and log-majorization uses partial products \cite{HardyLittlewoodPolya1952,MarshallOlkinArnold2011}. Applied to eigenvalue or singular-value lists, these orders control leading eigenvalues, traces, determinants, Ky Fan sums, and unitarily invariant norms even when a direct L\"owner-order comparison of the original matrices is unavailable \cite{Bhatia1997}. This flexibility is important in quantum theory: trace-convexity results underpin strong subadditivity and data processing, while log-majorization controls noncommutative operator means \cite{AudenaertDatta2015,Hiai2024,LiebRuskai1973StrongSubadditivity}. Majorization inequalities therefore provide a common framework for proving monotonicity under mixing, quantum channels, and resource conversion.

Unitary-orbit inequalities sharpen spectral comparisons into operator statements that remain useful under subsequent quantum processing. Thompson proved that \cite{Thompson1976ConvexConcaveSingularValues}, for arbitrary matrices $X$ and $Y$, there exist unitaries $U,V$ such that
\begin{equation}
|X+Y|\leq U|X|U^\dagger+V|Y|V^\dagger.
\end{equation}
For an operator convex function $f$, Jensen's operator inequality gives $f(\Phi(A)) \leq \Phi(f(A))$ for a unital positive map $\Phi$ under the usual hypotheses \cite{HansenPedersen2003Jensen}, and Aujla--Bourin and Bourin--Lee established Jensen-type counterparts of Thompson: for a unital positive map $\Phi$ and a monotone convex function $f$, one has $f(\Phi(A))\leq U\Phi(f(A))U^\dagger$ for some unitary $U$; for a general convex $f$, one still has the inequality $f(\Phi(A))\leq\bigl(U\Phi(f(A))U^\dagger+V\Phi(f(A))V^\dagger\bigr)/2$ for suitable unitaries $U,V$ \cite{AujlaBourin2007EigenvalueConvexLogConvex,BourinLee2012UnitaryOrbits}. These results replace the restrictive assumption of operator convexity by scalar convexity together with unitary changes of basis. The gain over weak majorization is operational: weak majorization yields Ky Fan and unitarily invariant norm bounds, but it does not compare the expectation of a prescribed observable or provide an operator inequality that can be passed through another quantum operation. By contrast, a bound of the form $\mathcal{T}(\sum_t\alpha_tz_t)\leq\sum_t\alpha_tU_t\mathcal{T}(z_t)U_t^\dagger$ is preserved by every positive map $\Psi$; after applying $\Psi$, every POVM effect $0\leq M\leq\one$ gives the corresponding expectation-value inequality. Since the Heisenberg adjoint $\mathcal{N}^\dagger$ of a quantum channel is unital and completely positive, the Jensen-type examples apply directly to nonlinear observables such as $f(\mathcal{N}^\dagger(H))$. A unitary-orbit lift therefore retains the full matrix before trace, entropy, or divergence scalarization and may furnish an operator-level precursor or refinement of an entropic inequality. The structural question is then how to recognize, from spectral information, when such a lift exists.

Horn inequalities answer this question exactly: For Hermitian matrices $C,A_1,\ldots,A_n$, the existence of unitaries $U_1,\ldots,U_n$ such that
\begin{equation}
C\leq \sum_{t=1}^{n}U_tA_tU_t^\dagger
\end{equation}
is equivalent to the complete family of inequalities
\begin{equation}
\sum_{k\in K}\lambda_k^\downarrow(C)
\leq
\sum_{t=1}^{n}\sum_{i\in I_t}\lambda_i^\downarrow(A_t)
\end{equation}
for all Horn tuples $(I_1,\ldots,I_n;K)$ \cite{Fulton2000EigenvaluesInvariantFactors,Fulton2000Eigenvalues,MR4529925,Horn1962,Klyachko1998,KnutsonTao1999}; see \cref{theorem: Fulton's Horn inequalities}. The initial-index cases are precisely the Ky Fan inequalities visible to weak majorization, whereas the complete Horn system controls arbitrary admissible selections of eigenvalues. Thus, a complete Horn lift does more than refine a spectral inequality: it certifies a direct L\"owner-order comparison on unitary orbits. This characterization provides the mechanism needed to promote majorization bounds to operator inequalities suitable for quantum-information applications.

The main contribution of this article is a general mechanism for lifting majorization inequalities to the complete Horn system and hence, in the additive setting, to unitary-orbit inequalities. The framework consists of a geometric reduction followed by a family-specific analytic verification. For the geometric reduction, we prove a generalized Hersch--Zwahlen variational formula (\cref{theorem: generalized Hersch-Zwahlen}) that represents coordinatewise monotone functions of arbitrary selected eigenvalue sublists as extrema of scalarizers over Schubert-constrained compressions, and combine it with the intersection property of Horn tuples to establish a lifting principle (\cref{thm: lifting principle}); this reduces each Horn-level extension to the family-specific analytic task of verifying the appropriate concavity, convexity, or mean inequality for the compressed scalarizer (\cref{prop:unified lifting}). As examples we apply this framework to three families that are central to the preceding quantum motivation. For the Aujla--Silva weak-majorization Jensen inequality, it yields complete Horn and unitary-orbit Jensen bounds for nonlinear mixtures under ordinary scalar convexity \cite{Aujla2003WeakMajorization}. For the Carlen--Frank--Lieb--Zhang positive-power family, it lifts the trace concavity and convexity results underlying $\alpha$-$z$ R\'{e}nyi data processing to selected-eigenvalue and unitary-orbit inequalities for the corresponding sandwiched-power matrices \cite{AudenaertDatta2015,CarlenFrankLieb2016SomeOperatorTraceConvexity,CarlenLieb2008MinkowskiTraceII,Lieb1973,Zhang2020WYDCarlenFrankLieb}. For geometric and multivariable geodesic means, it gives multiplicative Horn refinements of log-majorization and, for Karcher means---including the weighted geometric mean---unitary-orbit inequalities after taking logarithms \cite{AndoHiai1994,Hiai2024,KuboAndo1980Means}. These results preserve matrix-level information that is lost in trace, norm, determinant, or entropy scalarizations and may therefore serve as operator-level antecedents or refinements of familiar quantum-information inequalities.

The remainder of the article is organized as follows. Section~II collects the preliminaries for the lifting arguments, including the complete Horn inequalities and the analytic inputs needed for the relevant concavity and convexity statements. Section~III develops the lifting framework: the generalized Hersch--Zwahlen variational formula in \cref{theorem: generalized Hersch-Zwahlen} expresses coordinatewise monotone functions of selected eigenvalues through scalarizers of matrix compressions, and the lifting principle in \cref{prop:unified lifting} reduces the Horn-level problem to establishing concavity or convexity of the corresponding scalarizer on every compressed subspace. Section~IV applies this framework to the three families just described, yielding the Horn-level extensions in \cref{theorem: Aujla-Silva Horn}, \cref{theorem: Two-variable positive-map powers}, and \cref{theorem: Horn inequalities for geodesic means} together with their corollaries.

\section{Preliminaries}
Let $\bM_d^+$ denote the set of $d\times d$ positive semidefinite matrices, and $\bP_d$ denote the set of $d\times d$ positive definite matrices. 

\subsection{Horn inequalities}
Fix $1\leq r\leq d$. For a subset
$I=\{i_1<...<i_r\}\subset[d]:=\{1,\dots,d\}$, define the dual subset
\begin{align}
    I^\vee:=\{d+1-i_r<...<d+1-i_1\}\subset[d].
\end{align}
If $E_\bullet=(E_1\subset\cdots\subset E_n=\bC^n)$ is a complete flag, define the Schubert variety
\begin{align}
    \Omega_I(E_\bullet):=\{L\in\Gr(r,n):\dim(L\cap E_{i_q})\ge q\text{ for }q=1,\dots,r\}.
\end{align}

\begin{defn}
An $(n+1)$-tuple $(I_1,\dots,I_n;K)$ of $r$-subsets of $[d]$ is called a \emph{Horn tuple} of rank $r$ in dimension $d$ if
\begin{align}
    \Omega_{I_1^\vee}(E_\bullet^{(1)})\cap\cdots\cap \Omega_{I_n^\vee}(E_\bullet^{(n)})\cap \Omega_K(G_\bullet)\neq\varnothing
\end{align}
for every choice of complete flags $(E_\bullet^{(1)},\dots,E_\bullet^{(n)},G_\bullet)$ in $\bC^d$.
\end{defn}

The Horn inequalities provide an equivalence between Hermitian matrix inequalities in Lowner order and eigenvalue inequalities over Horn tuples:

\begin{thm}\cite{Fulton2000Eigenvalues} \label{theorem: Fulton's Horn inequalities}
Write $\alpha^{(t)}=(\alpha^{(t)}_1\geq \cdots \geq \alpha^{(t)}_d)$ for $1\leq t\leq n$, and $\gamma=(\gamma_1\geq \cdots \geq \gamma_d)$. Then the following statements are equivalent:\\
(1) There exist Hermitian $d\times d$ matrices $A_1,\dots,A_n,C$ with eigenvalues
$\alpha^{(1)},\dots,\alpha^{(n)},\gamma$, respectively, such that
\begin{align}
    C\leq A_1+\cdots +A_n.
\end{align}
(2) For every Horn tuple $(I_1,\cdots,I_n;K)$, one has
\begin{align}
    \sum_{k\in K}\gamma_k \leq \sum_{t=1}^n \sum_{i_t\in I_t} \alpha_{i_t}^{(t)}.
\end{align}
\end{thm}

\subsection{Analytic inputs for power family}\label{sec:analytic}
We start with Zhang's theorem on the joint concavity / convexity regions of $\Psi_{p,q,s}^K(A,B)$:

\begin{thm}\cite[Theorem 1.1]{Zhang2020WYDCarlenFrankLieb} \label{thm: Zhang's concavity}
    For a fixed invertible matrix \(K\), define
\begin{align}
\Psi_{p,q,s}^K(A,B)
:=\tr\left(B^{q/2}K^*A^p K B^{q/2}\right)^s.
\end{align}
The sharp joint concavity / convexity regions of this function, symmetrically in \(p\) and \(q\), are given by:\\
(1) \(\Psi_{p,q,s}^K\) is jointly concave if \(0\le p,q\le1\) and \(0<s\le1/(p+q)\).\\
(2) \(\Psi_{p,q,s}^K\) is jointly convex if \(-1\le p,q\le0\) and \(s>0\).\\
(3)\(\Psi_{p,q,s}^K\) is jointly convex in the mixed region with one exponent in \([1,2]\), the other in \([-1,0]\), and \(s\ge1/(p+q)\), excluding the endpoint where \(p+q=0\).
\end{thm}

We will also use the variational identity
\begin{equation}\label{eq:zhang-var}
\tr|XY|^{r_1}
=
\max_Z\left\{
\frac{r_1}{r_0}\tr|XZ|^{r_0}
-
\frac{r_1}{r_2}\tr|Y^{-1}Z|^{r_2}
\right\},
\qquad
\frac1{r_0}=\frac1{r_1}+\frac1{r_2},
\end{equation}
for positive \(r_0,r_1,r_2\) \cite[(3.2)]{Zhang2020WYDCarlenFrankLieb}. In the rest of this subsection, we prove some analytic inputs which will be used for applying the lifting framework to the power family in \cref{theorem: Two-variable positive-map powers}. 

\begin{lem} \label{lem: order}
Given strictly positive map $\Phi$, \(0\le p\le q\le1\), and matrix $K$.\\
(1) $X\mapsto \Phi(X^{-p})^{-1}$ is operator increasing and operator concave \cite[Lemma 2.4]{Hiai2012ConcavityMatrixTraceNorm}.\\
(2) $g_{p,q}(X):=\tr\bigl[K^*\Phi(X^{-p})^{-1}K\bigr]^{1/q}$ is concave.
\end{lem}
\begin{proof}
    (2) By \cite[Corollary 4.5]{Hiai2016ConcavityII}, $f_p(X):=\tr(\Psi(X^{-p})^{-\frac{1}{p}})$ is concave for strictly positive map $\Psi$. If $K>0$, apply this to $\Psi(X):=K^{-1}\Phi(X)K^{-1}$ proves the claim. For generic $K$, set $K_\epsilon=(KK^*)^\frac{1}{2}+\epsilon\one$ and taking $\epsilon\to0$ proves the claim by continuity. Now set $\alpha=\frac{p}{q}\le 1$ and $X_\lambda=\lambda X_1+(1-\lambda)X_2$, then
    \begin{align}
        g_{p,q}(X)=f_q(X_\lambda^\alpha)\geq f_q(\lambda X_1^\alpha+(1-\lambda)X_2^\alpha)&\geq\lambda f_q(X_1^\alpha)+(1-\lambda)f_q(X_2^\alpha)\\
        &=\lambda g_{p,q}(X_1)+(1-\lambda)g_{p,q}(X_2).
    \end{align}
\end{proof}

\begin{lem} \label{lem:shorted-epstein}
Let \(V:\bC^r\to\bC^d\) be an isometry.\\
(1) If \(0\le p,q\le1\) and \(0\le s\le\frac{1}{p+q}\), then
\begin{align}
f(X,Y):=\tr\left[
\left(
(V^*\Phi(X^p)^{-1}V)^{-\frac{1}{2}}
V^*\Psi(Y^q)V
(V^*\Phi(X^p)^{-1}V)^{-\frac{1}{2}}
\right)^s
\right]
\end{align}
is jointly concave.\\
(2) If \(-1\le p,q\le0\) and \(0\le s\le -\frac{1}{p+q}\), then
\begin{align}
g(X,Y)=\tr\left[
\left(
(V^*\Phi(X^p)V)^{-\frac{1}{2}}
V^*\Psi(Y^q)^{-1}V
(V^*\Phi(X^p)V)^{-\frac{1}{2}}
\right)^s
\right]
\end{align}
is jointly concave.
\end{lem}

\begin{proof}
    (1) Let $W_0:=A^{-1}V(V^*A^{-1}V)^{-1}$. If $V^*W=\one_r$, then \begin{align}
        W^*AW=W_0^*AW_0+(W-W_0)^*A(W-W_0)\geq W_0^*AW_0=(V^*A^{-1}V)^{-1}
    \end{align}
    with equality when $W=W_0$, therefore in Loewner order $(V^*A^{-1}V)^{-1}=\inf_{W:V^*W=\one_r}W^*AW$. By \cite[Theorem 2.1]{Hiai2016ConcavityII}, the map $h_W(X,Y):=\tr\left[\left((W^*\Phi(X^p)W)^{\frac{1}{2}}V^*\Psi(Y^q)V(W^*\Phi(X^p)W)^{\frac{1}{2}}\right)^s\right]$ is jointly concave for any fixed $W$, so $f(X,Y)=\inf_W h_W(X,Y)$ is also jointly concave.\\
    (2) By Choi-Davis-Jensen's inequality, $F(T)^{-1}\leq F(T^{-1})$ for any unital positive linear map $F$ and $T>0$. Given any $R>0$ and positive linear map $\Theta$, applying Choi-Davis-Jensen to $T\coloneqq R^{-\frac{1}{2}}ZR^{-\frac{1}{2}}$ and $F(Z)\coloneqq\Theta(R^{-1})^{-\frac{1}{2}}\Theta(R^{-\frac{1}{2}}ZR^{-\frac{1}{2}})\Theta(R^{-1})^{-\frac{1}{2}}$ we get 
    \begin{align}
        \Theta(Z^{-1})^{-1}=\inf_{R>0}G_{R,\Theta}(Z);\qquad G_{R,\Theta}(Z)\coloneqq\Theta(R^{-1})^{-1}\Theta(R^{-1}ZR^{-1})\Theta(R^{-1})^{-1}.    
    \end{align}
    Now set $\Phi_R(Z)\coloneqq G_{R,V^*\Phi V}(Z)$ and $\Psi_S(Z)\coloneqq V^*G_{S,\Psi}(Z)V$, then $(V^*\Phi(X^p)V)^{-1}=\inf_{R>0}\Phi_R(X^{-p})$ and $V^*\Psi(Y^q)^{-1}V=\inf_{S>0}\Psi_S(Y^{-q})$. By part (1), $\tilde{g}_{R,S}(X,Y)\coloneqq\tr[(\Phi_R(X^{-p})^\frac{1}{2}\Psi_S(Y^{-q})\Phi_R(X^{-p})^\frac{1}{2})^s]$ is jointly concave for all fixed $R,S>0$, therefore so is $g(X,Y)=\inf_{R,S>0}\tilde{g}_{R,S}(X,Y)$.
\end{proof}

\begin{lem}\label{lem:engine I}
Let \(P,Q:\cE\to\bP_r\) be operator-concave maps on a convex set \(\cE\).  If \(s\le0\), then
\begin{align}
z\mapsto G_s(P(z),Q(z))
:=\tr(P(z)^{1/2}Q(z)P(z)^{1/2})^s
\end{align}
is convex.
\end{lem}
\begin{proof}
Concavity gives \(P(\bar z)\ge\sum_t\alpha_tP(z_t)\) and \(Q(\bar z)\ge\sum_t\alpha_tQ(z_t)\).  Since \(G_s\) is decreasing in each variable and jointly convex,
\begin{align}
G_s(P(\bar z),Q(\bar z))
\le
G_s\left(\sum_t\alpha_tP(z_t),\sum_t\alpha_tQ(z_t)\right)
\le
\sum_t\alpha_tG_s(P(z_t),Q(z_t)).
\end{align}
\end{proof}

\begin{lem}\label{lem: PL}
For every positive linear map \(\Phi\), every \(1\le q\le2\), and every \(t\ge\frac{1}{q}\), $X\mapsto\tr\Phi(X^q)^t$ is convex on the strictly positive cone.
\end{lem}
\begin{proof}
    Suppose $t\geq 1$, set $X_\lambda=\lambda X_1+(1-\lambda)X_2$, then
    \begin{align}
        \tr\Phi(X_\lambda^q)^t\leq\tr(\lambda\Phi(X_1^q)+(1-\lambda)\Phi(X_2^q))^t\leq\lambda\tr\Phi(X_1^q)^t+(1-\lambda)\tr\Phi(X_2^q)^t.
    \end{align}
    Now suppose $\frac{1}{q}\leq t<1$, by \cite[Lemma 2.2]{CarlenLieb2008MinkowskiTraceII},
    \begin{align}
        \tr\Phi(X^q)^t=t\inf_{Y>0}\{\tr(\Phi(X^q)Y^{1-\frac{1}{t}})+(\frac{1}{t}-1)\tr(Y)\}.
    \end{align}
    Applying Zhang's variational formula \eqref{eq:zhang-var} with $r_1=2$, $r_0=\frac{2}{q}$, $r_2=\frac{2}{q-1}$
    gives
    \begin{align}
        \tr(\Phi(X^q)Y^{1-\frac{1}{t}})=\tr|X^\frac{q}{2}\Phi^*(Y^{1-\frac{1}{t}})^\frac{1}{2}|^2=\sup_Z\{q\tr(Z^*X^qZ)^\frac{1}{q}-(q-1)\tr(Z^*\Phi^*(Y^{1-\frac{1}{t}})^{-1}Z)^\frac{1}{q-1}\}.
    \end{align}
    For any fixed $Z$, the first term is convex by \cref{thm: Zhang's concavity} and the second term is concave by \cref{lem: order}(2). Therefore $\tr(\Phi(X^q)Y^{1-\frac{1}{t}})$ is jointly convex and so is $\tr\Phi(X^q)^t$.
\end{proof}

\begin{lem}\label{lem:engine II}
Let $p=-a$, $0<a\le1$, $1\le q\le2$, $s\ge\frac1{q-a}$. Given isometry $V$, define
\begin{align}
F(X,Y):=\tr\left((V^*\Phi(X^{-a})^{-1}V)^{-\frac{1}{2}}V^*\Psi(Y^q)V(V^*\Phi(X^{-a})^{-1}V)^{-\frac{1}{2}}\right)^s.
\end{align}
Then \(F\) is jointly convex.
\end{lem}

\begin{proof}
Write $F(X,Y)=\tr|(V^*\Psi(Y^q)V)^\frac{1}{2}(V^*\Phi(X^{-a})^{-1}V)^{-\frac{1}{2}}|^{2s}$ and $\frac1t=\frac1s+a$, then \(t\ge1/q\). Applying Zhang's variational formula \eqref{eq:zhang-var} with $r_1=2s$, $r_0=2t$, $r_2=\frac{2}{a}$ gives
\begin{equation}\label{eq:mixed-var}
F(X,Y)=
\max_Z\left\{
\frac{s}{t}\tr\bigl(Z^*V^*\Psi(Y^q)VZ\bigr)^t
-sa\,\tr\bigl(Z^*(V^*\Phi(X^{-a})^{-1}V)Z\bigr)^{1/a}
\right\}.
\end{equation}
For fixed \(Z\), the first term is convex in \(Y\) by \cref{lem: PL}, and the second term is concave by \cref{lem: order}(2). Therefore $F$ is convex as a maximum of convex functions.
\end{proof}

\subsection{The geometric means}
\subsubsection{Barycenter \cite{Sturm2003ProbabilityMeasuresNPC}}
A Hadamard space, also called a global NPC (nonpositive curvature) space, is a complete metric space
in which every pair \(x_0,x_1\) has a midpoint \(m\) satisfying
\begin{equation}\label{eq:npc}
  d(z,m)^2
  \le \frac12 d(z,x_0)^2+\frac12 d(z,x_1)^2
      -\frac14 d(x_0,x_1)^2
  \qquad (z\in N).
\end{equation}
The curvature condition implies that geodesics are unique. A function \(f:N\to\bR\) is geodesically concave if, along every geodesic,
\begin{align}
      f(\gamma(t))\ge(1-t)f(\gamma(0))+t f(\gamma(1)),\qquad 0\le t\le1.
\end{align}

For a finitely supported probability measure on a Hadamard space \((N,d)\), $\mu=\sum_{i=1}^n w_i\delta_{x_i}$ where \(\delta_x\) denotes the unit point mass at \(x\), define its barycenter by
\begin{equation}\label{eq:barycenter}
  \bc(\mu)
  =\argmin_{z\in N}\sum_{i=1}^n w_i d(z,x_i)^2.
\end{equation}
The minimizer exists and is unique. We will use the following standard result:

\begin{thm}[Jensen's inequality for barycenters]\label{thm:jensen for barycenters}
Let \((N,d)\) be a Hadamard space,
\(\mu=\sum_iw_i\delta_{x_i}\). If \(f\) is lower semicontinuous and geodesically concave, then
\begin{equation}\label{eq:jensen-concave}
  f(bc(\mu))\ge\sum_iw_i f(x_i).
\end{equation}
\end{thm}

\subsubsection{Geodesic mean \cite{BourinHiai2014JensenMinkowski, Lim2013ConvexGeometricMeans}}
For \(X,Y\in\bP_d\) and \(0\le t\le1\), define the weighted geometric mean \cite{KuboAndo1980Means}
\begin{equation}\label{eq:geometric-mean}
  X\#_t Y:=X^{1/2}\bigl(X^{-1/2}YX^{-1/2}\bigr)^t X^{1/2}.
\end{equation}
We will use the immediate properties
\begin{align}
  X\#_t Y\leq(1-t)X+tY;\qquad(X\#_t Y)^{-1}=X^{-1}\#_t Y^{-1};\qquad\det(A\#_t B)=\det(A)^{1-t}\det(B)^t
\end{align}
and Ando's inequality \cite{Ando1979ConcavityCertainMaps}:
\begin{thm} \label{thm: Ando's inequality}
    For a positive unital linear map \(\Phi\), positive definite \(A,B\), and \(0\le t\le1\),
    \begin{equation}\label{eq:ando}
        \Phi(A\#_t B)\leq \Phi(A)\#_t\Phi(B).
    \end{equation}
    \\
\end{thm}

Define the distance function
\begin{align}
    \delta(X,Y):=||\log A^{-\frac{1}{2}}BA^{-\frac{1}{2}}||_2,
\end{align}
then \((\bP_d,\delta)\) is a Hadamard space with $\gamma(t)=X\#_t Y$ being the unique geodesic from \(X\)
to \(Y\) \cite{Bhatia2007PositiveDefiniteMatrices}. Given weight vector $w=(w_1,...,w_n)$ and $X_1,...,X_n\in\bP_d$, the weighted Karcher mean is defined by
\begin{equation}\label{eq:karcher-barycenter}
  G_n(w;X_1,\ldots,X_n)
  :=\bc\!\left(\sum_{i=1}^n w_i\delta_{X_i}\right)=\argmin_{M\in\bP_d}\sum_{i=1}^{n}w_i\delta^2(M,X_i)   .
\end{equation}
The $n-$variable geodesic mean is defined by
\begin{align}
    \sigma_\nu(X_1,...,X_n):=\int_{\Delta_n}G_n(w;X_1,...,X_n) d\nu(w)
\end{align}
where $\Delta_n:=\{w:w_i\geq0;\:\sum_iw_i=1\}$ is the $n-$simplex and $\nu$ is a probability measure on $\Delta_n$. When $n=2$, the Karcher mean recovers the geometric mean $G_2\big((1-t,t);X,Y\big)=X\#_t Y$
and $\sigma_2$ recovers the two-variable geodesic mean $\sigma_2(X,Y)=X\sigma Y:=\int_0^1 X\#_t Yd\nu(t)$.

\section{Abstract framework for lifting}
\subsection{Generalized Hersch-Zwahlen variational formula}
Given $X\in\Herm(d)$, denote $F_\bullet^\uparrow(X)$ and $F_\bullet^\downarrow(X)$ be an increasing and decreasing eigenflag of $X$ respectively.
For any $L\in\Gr(r,d)$, let $V_L:\bC^r\to\bC^d$ be an isometry such that $\im(V_L)=L$, define $X_L:=V_L^\dagger XV_L$. A function $\varphi:\bR^r\to\bR$ is called coordinatewise nondecreasing if $x_i\leq y_i\:\forall\:i$ implies $\varphi(x)\leq\varphi(y)$ and is called coordinatewise nonincreasing if $x_i\leq y_i\:\forall\:i$ implies $\varphi(x)\geq\varphi(y)$. Define $\Theta_{X,\varphi}(L):=\varphi(\lambda_1^\uparrow(X_L),...,\lambda_r^\uparrow(X_L))$.

\begin{thm}[Generalized Hersch-Zwahlen] \label{theorem: generalized Hersch-Zwahlen}
    Given $I=\{i_1<...<i_r\}\subset[d]$, for any  $X\in\Herm(d)$ and $M\in\GL(d,\bC)$ we have
    \begin{align}
        \lambda_q^\uparrow(X_{M^{-1}L})\leq\lambda^\uparrow_{i_q}(X)\quad\forall\:L\in\Omega_I(MF_\bullet^\uparrow(X));\:\forall\:q\in[r].
    \end{align}
    \begin{align}
        \lambda_q^\uparrow(X_{M^{-1}L})\geq\lambda^\uparrow_{i_q}(X)\quad\forall\:L\in\Omega_{I^\vee}(MF_\bullet^\downarrow(X));\:\forall\:q\in[r].
    \end{align}
    Consequently:\\
    (1) If $\varphi$ is coordinatewise nondecreasing, then
    \begin{align}
        \varphi(\lambda_{i_1}^\uparrow(X),...,\lambda_{i_r}^\uparrow(X))=\max_{L\in\Omega_I(MF_\bullet^\uparrow(X))}\Theta_{X,\varphi}(M^{-1}L)=\min_{L\in\Omega_{I^\vee}(MF_\bullet^\downarrow(X))}\Theta_{X,\varphi}(M^{-1}L).
    \end{align}
    (2) If $\varphi$ is coordinatewise nonincreasing, then
    \begin{align}
        \varphi(\lambda_{i_1}^\uparrow(X),...,\lambda_{i_r}^\uparrow(X))=\min_{L\in\Omega_I(MF_\bullet^\uparrow(X))}\Theta_{X,\varphi}(M^{-1}L)=\max_{L\in\Omega_{I^\vee}(MF_\bullet^\downarrow(X))}\Theta_{X,\varphi}(M^{-1}L).
    \end{align}
\end{thm}
\begin{proof}
    If $L\in\Omega_I(MF_\bullet^\uparrow(X))$, then $M^{-1}L\in\Omega_I(F_\bullet^\uparrow(X))$. For each $q\in[r]$ choose a $q-$dimensional subspace $V_q\subset M^{-1}L\cap F_{i_q}^\uparrow(X)$, then $\langle Xv,v\rangle\leq\lambda_{i_q}^\uparrow(X)\:\forall\:v\in V_q$. By Courant-Fischer minimax principle,
    \begin{align}
        \lambda_q^\uparrow(X_{M^{-1}L})\leq\max_{v\in V_q:\:||v||=1}\langle Xv,v\rangle\leq\lambda_{i_q}^\uparrow(X).
    \end{align}
    For attainment, choose an ordered eigenbasis $u_1,\ldots,u_d$ or $X$ and set $L_I=M\Span\{u_{i_1},\ldots,u_{i_r}\}$, then $X_{M^{-1}L_I}$ has eigenvalues $\lambda^\uparrow_{i_1}(X),\ldots,\lambda^\uparrow_{i_r}(X)$.
    
    If $L\in\Omega_{I^\vee}(MF_\bullet^\downarrow(
    X
    ))$, then $M^{-1}L\in\Omega_{I^\vee}(F_\bullet^\downarrow(X))$. For each $q\in[r]$ choose an $(r+1-q)$ dimensional subspace $W_q\subset M^{-1}L\cap F_{d+1-i_q}^\downarrow(X)$, then $\langle Xw,w\rangle\geq\lambda_{d+1-i_q}^\downarrow(X)=\lambda_{i_q}^\uparrow(X)\:\forall\:w\in W_q$. By Courant-Fischer principle again,
    \begin{align}
        \lambda_q^\uparrow(X_{M^{-1}L})\geq\min_{w\in W_q:\:||w||=1}\langle Xw,w\rangle\geq\lambda_{i_q}^\uparrow(X).
    \end{align}
\end{proof}

\begin{cor}[Linear Hersch-Zwahlen]
    For any $X\in\Herm(d)$ and $I\subset[d]$,
    \begin{align}
        \sum_{i\in I}\lambda_i^\uparrow(X)=\max_{L\in\Omega_I(F_\bullet^\uparrow(X))}\tr(X_L)=\min_{L\in\Omega_{I^\vee}(F_\bullet^\downarrow(X))}\tr(X_L).
    \end{align}
    This is the standard Hersch-Zwahlen variational formula which plays a key role in the proof of the Horn's conjecture \cite{Fulton2000EigenvaluesInvariantFactors}.
\end{cor}

\subsection{Lifting principle}
\begin{thm}\label{thm: lifting principle}
    Given an $n-$tuple operation $\Phi:\mc X^n\to\mc X$, a function $\Gamma:\bR^n\to\bR$ that is nondecreasing in each variable, together with\\
    (1) for each $I\subset[d]$, a map $\Lambda_I:\mc X\to\bR$;\\
    (2) for each $L\in\Gr(r,d)$, a map $\Theta_L:\mc X\to\bR$;\\
    (3) for each $X\in\mc X$, two complete flags $F^-_\bullet(X)$ and $F^+_\bullet(X)$.\\
    Assume the following axioms:\\
    (L1) If $L\in\Omega_{I^\vee}(F^-_\bullet(X))$, then $\Lambda_I(X)\leq\Theta_L(X)$; If $L\in\Omega_I(F^+_\bullet(X))$, then $\Theta_L(X)\leq\Lambda_I(X)$.\\
    (L2) $\Theta_L(\Phi(X_1,...,X_n))\geq\Gamma(\Theta_L(X_1),...,\Theta_L(X_n))\:\forall\:L\in\Gr(r,d)$.\\
    Then for every Horn triple $(I_1,...,I_n;K)$ in dimension $d$,
    \begin{align}
        \Lambda_K(\Phi(X_1,...,X_n))\geq\Gamma(\Lambda_{I_1}(X_1),...,\Lambda_{I_n}(X_n)).
    \end{align}
\end{thm}
\begin{proof}
    Choose $L\in\Omega_{I_1^\vee}(F^-_\bullet(X_1))\cap\cdots\cap \Omega_{I_n^\vee}(F^-_\bullet(X_n))\cap \Omega_K(F^+_\bullet(Y))$, then the axioms immediately imply
    \begin{align}
        \Lambda_K(\Phi(X_1,...,X_n))\geq\Theta_L(\Phi(X_1,...,X_n))\geq\Gamma(\Theta_L(X_1),...,\Theta_L(X_n))\geq\Gamma(\Lambda_{I_1}(X_1),...,\Lambda_{I_n}(X_n)).
    \end{align}
\end{proof}

The following is a concrete lifting principle which reduces the lifting to a concavity condition for a scalarizer function $\Theta_L$. In the next section, we will apply this principle to generalize the majorization inequalities by Aujla-Silva in \cref{theorem: Aujla-Silva Horn}, Carlen-Frank-Lieb-Zhang in \cref{theorem: Two-variable positive-map powers}, and geometric means in \cref{theorem: Horn inequalities for geodesic means}.

\begin{prop}\label{prop:unified lifting}
Let $\cE$ be a convex set, $Z:\cE\to\Herm(d)$ and $M:\cE\to GL(d,\bC)$ be two maps, and $\varphi:\bR^r\to\bR$ be coordinatewise nondecreasing. For every $r$-plane $L\in\Gr(r,d)$ define
\begin{align}
    \Theta_L(z):=\Theta_{Z(z),\varphi}\bigl(M(z)^{-1}L\bigr).
\end{align}
For $I=\{i_1<\cdots<i_r\}$ put
\begin{align}
    \Lambda_I(z):=\varphi\bigl(\lambda^\uparrow_{i_1}(Z(z)),\ldots,\lambda^\uparrow_{i_r}(Z(z))\bigr).
\end{align}
If $\Theta_L$ is concave for every $L\in\Gr(r,d)$, then for every Horn tuple $(I_1,\ldots,I_n;K)$ and every probability distribution $\{\alpha_t:1\leq t\leq n\}$,
\begin{align}\label{eq:abstract-horn}
    \Lambda_K\left(\sum_{t=1}^n\alpha_tz_t\right)\ge\sum_{t=1}^n\alpha_t\Lambda_{I_t}(z_t).
\end{align}
If $\varphi$ is coordinatewise nonincreasing, then put $\Lambda_I(z)\coloneqq\varphi(\lambda^\uparrow_{I^\vee}(Z(z)))$ and the claim still holds.
\end{prop}

\begin{proof}
\cref{theorem: generalized Hersch-Zwahlen} gives (L1) in \cref{thm: lifting principle}:
\begin{align}
    \Lambda_I(z)=\max_{L\in\Omega_I(M(z)F_\bullet^\uparrow(Z(z)))}\Theta_L(z)=\min_{L\in\Omega_{I^\vee}(M(z)F_\bullet^\downarrow(Z(z)))}\Theta_L(z).
\end{align}
The assumed concavity gives (L2) with $\Gamma(x_1,\ldots,x_n)=\sum_t\alpha_tx_t$. Therefore \cref{thm: lifting principle} gives \eqref{eq:abstract-horn} with $\Phi(z_1,\ldots,z_n)=\sum_t\alpha_t z_t$.
\end{proof}

\section{Additive Horn families}
\subsection{Convex function family}
In this section, we apply the lifting principle to the weak majorization inequalities for convex functions in \cite{Aujla2003WeakMajorization}. We first prove a general additive Hersch-Zwahlen formula and lifting lemma:

\begin{lem}[Additive Hersch-Zwahlen] \label{lemma: ordinary nonlinear HZ}
    Given nondecreasing function $h:\bR\to[0,\infty)$ and $I\subset[d]$ with $|I|=r$, for any $X\in\Herm(d)$ we have
    \begin{align}
    \sum_{i\in I} h\bigl(\lambda_i^\uparrow(X)\bigr) = \max_{L\in \Omega_I(F_\bullet^\uparrow(X))} \tr\:h(X_L) = \min_{L\in \Omega_{I^\vee}(F_\bullet^\downarrow(X))} \tr\:h(X_L).
    \end{align}
    If $h$ is nonincreasing, then
    \begin{align}
        \sum_{i\in I} h\bigl(\lambda_i^\uparrow(X)\bigr) = \min_{L\in \Omega_I(F_\bullet^\uparrow(X))} \tr\:h(X_L) = \max_{L\in \Omega_{I^\vee}(F_\bullet^\downarrow(X))} \tr\:h(X_L).
    \end{align}
\end{lem}
\begin{proof}
    Let $\varphi_h(x_1,...,x_r):=\sum_{j=1}^{r}h(x_j)$ and $M=\one$, since $h$ is nondecreasing, $\varphi_h$ is coordinatewise nondecreasing and $\Theta_{X,\varphi_h}(L)=\tr h(X_L)$, by \cref{theorem: generalized Hersch-Zwahlen},
    \begin{align}
    \sum_{i\in I} h\bigl(\lambda_i^\uparrow(X)\bigr) = \varphi_h\big(\lambda_{i_1}^\uparrow(X),...,\lambda_{i_r}^\uparrow(X)\big) = \max_{L\in \Omega_I(F_\bullet^\uparrow(X))} \tr\:h(X_L) = \min_{L\in \Omega_{I^\vee}(F_\bullet^\downarrow(X))} \tr\:h(X_L).
    \end{align}
\end{proof}

\begin{lem} \label{lemma: compressed-trace Horn}
Given $F:\bM_d\to\bM_{d'}$, let $I_F$ be an interval containing $\bigcup_{X\in\bM_d^+}\spec(F(X))$, $h:I_F\to\bR$ be nondecreasing and $T(X):=h\bigl(F(X)\bigr)$. Then for every $X_1,...,X_n\in \bM_d^+$ and every probability distribution $\{\alpha_t:1\leq t\leq n\}$:\\
(1) If $X\mapsto\tr\:h(F(X)_L)$ is concave for all $L$, then for every Horn tuple $(I_1,...,I_n;K)$ in dimension $d'$,
\begin{align}
\sum_t\alpha_t\sum_{i_t\in I_t}\lambda_{i_t}^\uparrow\bigl(T(X_t)\bigr)\leq\sum_{k\in K}\lambda_k^\uparrow\bigl(T(\sum_t\alpha_tX_t)\bigr).
\end{align}
Equivalently, there exist unitaries $U_t$ such that 
    \begin{align}
        T[\sum_t\alpha_tX_t]\geq\sum_t\alpha_tU_tT(X_t)U_t^\dagger.
    \end{align}
(2) If $X\mapsto\tr\:h(F(X)_L)$ is convex for all $L$, then
\begin{align}
    \sum_t\alpha_t\sum_{i_t\in I_t}\lambda_{i_t}^\downarrow\bigl(T(X_t)\bigr)\geq\sum_{k\in K}\lambda_k^\downarrow\bigl(T(\sum_t\alpha_tX_t)\bigr).
\end{align}
Equivalently, there exist unitaries $U_t$ such that 
    \begin{align}
        T[\sum_t\alpha_tX_t]\leq\sum_t\alpha_tU_tT(X_t)U_t^\dagger.
    \end{align}
\end{lem}
\begin{proof}
    Let $\varphi(x_1,\ldots,x_r):=\sum_j h(x_j)$, $Z(X):=F(X)$ and $M:=\one$, then $\Theta_L(X)=\tr\:h(F(X)_L)$ is concave by assumption, and $\Lambda_I(X)=\sum_{i\in I}h(\lambda_i^\uparrow(F(X)))=\sum_{i\in I}\lambda_i^\uparrow(T(X))$ since $h$ is nondecreasing, so \cref{prop:unified lifting} gives 
    \begin{align}
        \sum_t\alpha_t\sum_{i_t\in I_t}\lambda_{i_t}^\uparrow\bigl(T(X_t)\bigr)\leq\sum_{k\in K}\lambda_k^\uparrow\bigl(T(\sum_t\alpha_tX_t)\bigr).
    \end{align}
\end{proof}

Now we are consider convex functions and generalize the Ky Fan inequalities for convex functions in \cite{Aujla2003WeakMajorization} to unitary-orbit inequalities:

\begin{thm} \label{theorem: Aujla-Silva Horn}
    Given $X_t\in\bM_d^+$, probability distribution $\{\alpha_t\}$ and convex function $f:[0,\infty)\to[0,+\infty)$, 
    \begin{align}
        \sum_{k\in K}\lambda^\downarrow_k\big(f(\sum_{t=1}^{n}\alpha_tX_t)\big)\leq\sum_{t=1}^{n}\alpha_t\sum_{i_t\in I_t}\lambda^\downarrow_{i_t}\big(f(X_t)\big)
    \end{align}
    for all Horn tuple $(I_1,...,I_n;K)$ in dimension $d$. Equivalently, there exist unitaries $U_t$ such that
    \begin{align}
        f(\sum_t\alpha_tX_t)\leq\sum_t\alpha_tU_tf(X_t)U_t^\dagger.
    \end{align}
\end{thm}
\begin{proof}
    We first assume $f$ is nondecreasing. Let $h=f$, $F=\id$, then $\Theta_L(X)=\tr\:f(X_L)$ is convex since $f$ is convex. By \cref{lemma: compressed-trace Horn}, 
    \begin{align}
        \sum_{k\in K}\lambda^\downarrow_k\big(f(\sum_{t=1}^{n}\alpha_tX_t)\big)\leq\sum_{t=1}^{n}\alpha_t\sum_{i_t\in I_t}\lambda^\downarrow_{i_t}\big(f(X_t)\big).
    \end{align}
    Equivalently, there exist unitaries $U_t$ such that
    \begin{align}
        f(\sum_t\alpha_tX_t)\leq\sum_t\alpha_tU_t f(X_t)U_t^\dagger.
    \end{align}
    If $f$ is nonincreasing, let $h(t)=f(-t)$ and $F=-\id$, then $h$ is nondecreasing and $\Theta_L(X)=\tr\:f(X_L)$ is still convex, so the same argument applies.\\
    
    Now, for generic convex function $f$, let $R$ be the largest eigenvalue among $X_t$ and $x_0=\arg\min_{x\in[0,R]}f(x)$, define $f_+(x):=\begin{cases}
        0 & t\leq x_0\\
        f(x)-f(x_0) & t\geq x_0
    \end{cases}$ and $f_-(x)=\begin{cases}
        f(x)-f(x_0) & t\leq x_0\\
        0 & x\geq x_0
    \end{cases}$, then $f_+,f_-$ are both convex and monotone, and $f(x)=f(x_0)+f_+(x)+f_-(x)$. By monotonicity, we get unitary orbit inequalities for $f_+$ and $f_-$ respectively, putting them in two blocks we get
    \begin{align}
        (f_+\oplus f_-)(\sum_t\alpha_tX_t)\leq\sum_t\alpha_t(U_t^+\oplus U_t^-)\big(f_+(X_t)\oplus f_-(X_t)\big)(U_t^+\oplus U_t^-)^\dagger
    \end{align}
    where $X\oplus Y=\text{diag}(X,Y)$. Finally, $\sum_{i\in I}\lambda_i^\downarrow\big(f_+(X)\oplus f_-(X)\big)=\sum_{i\in I}\lambda_i^\downarrow(f(X))-|I|f(x_0)$ because $\{f_-(x),f_+(x)\}=\{0,f(x)-f(x_0)\}$ $\forall$ $x$, since Horn tuples in dimension $d$ are also Horn tuples in dimension $2d$, we still have
    \begin{align}
        \sum_{k\in K}\lambda^\downarrow_k\big(f(\sum_{t=1}^{n}\alpha_tX_t)\big)\leq\sum_{t=1}^{n}\alpha_t\sum_{i_t\in I_t}\lambda^\downarrow_{i_t}\big(f(X_t)\big).
    \end{align}
\end{proof}

\subsection{Two-variable positive power family}
In this section, we consider the following three parameter regions in Zhang's theorem \cite{Zhang2020WYDCarlenFrankLieb}:\\
Region one: $\{0\leq p,q\leq 1;\:0<s\leq\frac{1}{p+q}\}\cup\{-1\leq p,q\leq0;\:\frac{1}{p+q}\leq s<0\}$.\\
Region two: $\{0\leq p,q\leq 1;\:s<0\}\cup\{-1\leq p,q\leq0;\:s>0\}$.\\
Region three: $\{-1\leq p\leq 0;\:1\leq q\leq 2; s\geq\frac{1}{p+q}\}\cup\{-1\leq q\leq 0;\:1\leq p\leq 2;\:s\geq\frac{1}{p+q}\}$.

\begin{thm} \label{theorem: Two-variable positive-map powers}
    Let $\Phi:\bP_{d_1}\to\bP_d$ and $\Psi:\bP_{d_2}\to\bP_d$ be positive linear maps, denote 
    \begin{align}
        T(X,Y):=\bigg(\Phi(X^p)^\frac{1}{2}\Psi(Y^q)\Phi(X^p)^\frac{1}{2}\bigg)^s.    
    \end{align}
    For every $X_1,...,X_n\in\bP_{d_1}$, $Y_1,...,Y_n\in\bP_{d_2}$ and every probability distribution $\{\alpha_t:1\leq t\leq n\}$:\\
    (1) If $(p,q,s)$ is in Region one, then for every Horn tuple $(I_1,...,I_n;K)$ in dimension $d$,
    \begin{align}
        \sum_{k\in K}\lambda^\uparrow_k\big(T(\sum_{t=1}^{n}\alpha_tX_t,\sum_{t=1}^{n}\alpha_tY_t)\big)\geq\sum_{t=1}^{n}\alpha_t\sum_{i_t\in I_t}\lambda^\uparrow_{i_t}\big(T(X_t,Y_t)\big).
    \end{align}
    Equivalently, there exist unitaries $U_t$ such that
    \begin{align}
        T(\sum_t\alpha_tX_t,\sum_t\alpha_tY_t)\geq\sum_t\alpha_tU_t T(X_t,Y_t)U_t^\dagger.
    \end{align}
    (2) If $(p,q,s)$ is in Regions two or three, then for every Horn tuple $(I_1,...,I_n;K)$ in dimension $d$,
    \begin{align}
        \sum_{k\in K}\lambda^\downarrow_k\big(T(\sum_{t=1}^{n}\alpha_tX_t,\sum_{t=1}^{n}\alpha_tY_t)\big)\leq\sum_{t=1}^{n}\alpha_t\sum_{i_t\in I_t}\lambda^\downarrow_{i_t}\big(T(X_t,Y_t)\big).
    \end{align}
    Equivalently, there exist unitaries $U_t$ such that
    \begin{align}
        T(\sum_t\alpha_tX_t,\sum_t\alpha_tY_t)\leq\sum_t\alpha_tU_t T(X_t,Y_t)U_t^\dagger.
    \end{align}
\end{thm}
\begin{proof}
    Write $A(X):=\Phi(X^p)$, $B(Y):=\Psi(Y^q)$. For any $L$ choose isometry $V$ with $\im(V)=L$.\\
    (1) Set $\Lambda_I(X,Y)=\sum_{i\in I}\lambda_i^\uparrow\big(T(X,Y)\big)$.
    
    If $0\leq p,q\leq 1;\:0<s\leq\frac{1}{p+q}$: Set $M(X,Y):=A(X)^\frac{1}{2}$, $Z(X,Y):=A(X)^\frac{1}{2}B(Y)A(X)^\frac{1}{2}$ and $\varphi(x_1,\ldots,x_r)=\sum_j x_j^s$, then $\Theta_L(X,Y)=\tr\left[\left((V^*\Phi(X^p)^{-1}V)^{-\frac{1}{2}}V^*\Psi(Y^q)V(V^*\Phi(X^p)^{-1}V)^{-\frac{1}{2}}\right)^s\right]$. By \cref{lem:shorted-epstein}(1), $\Theta_L$ is jointly concave, therefore \cref{prop:unified lifting} proves the claim.

    If $-1\leq p,q\leq 0;\:\frac{1}{p+q}\leq s\leq 0$: Set $M(X,Y):=A(X)^{-\frac{1}{2}}$, $Z(X,Y):=(A(X)^\frac{1}{2}B(Y)A(X)^\frac{1}{2})^{-1}$ and $\varphi(x_1,\ldots,x_r)=\sum_j x_j^{-s}$, then $\Theta_L(X,Y)=\tr\left[\left((V^*\Phi(X^p)V)^{-\frac{1}{2}}V^*\Psi(Y^q)^{-1}V(V^*\Phi(X^p)V)^{-\frac{1}{2}}\right)^{-s}\right]$. By \cref{lem:shorted-epstein}(2), $\Theta_L$ is jointly concave, therefore \cref{prop:unified lifting} proves the claim.\\
    (2) Set $\Lambda_I(X,Y)=-\sum_{i\in I}\lambda_i^\downarrow\big(T(X,Y)\big)$.\\
    Region two: If $0\leq p,q\leq 1;\:s\leq0$: Set $M(X,Y):=A(X)^\frac{1}{2}$,  $Z(X,Y):=A(X)^\frac{1}{2}B(Y)A(X)^\frac{1}{2}$ and $\varphi(x_1,\ldots,x_r)=-\sum_j x_j^s$, then $\Theta_L(X,Y)=-\tr\left[\left((V^*A(X)^{-1}V)^{-\frac{1}{2}}V^*B(Y)V(V^*A(X)^{-1}V)^{-\frac{1}{2}}\right)^s\right]$. Since $0\leq p,q\leq 1$, both $A$ and $B$ are operator concave, then $X\mapsto(V^*A(X)^{-1}V)^{-1}$ is operator concave, so $\Theta_L$ is jointly concave by \cref{lem:engine I}, therefore \cref{prop:unified lifting} proves the claim.

    If $-1\leq p,q\leq 0;\:s\geq0$: Set $M(X,Y):=A(X)^{-\frac{1}{2}}$,  $Z(X,Y):=(A(X)^\frac{1}{2}B(Y)A(X)^\frac{1}{2})^{-1}$ and $\varphi(x_1,\ldots,x_r)=-\sum_j x_j^{-s}$, then $\Theta_L(X,Y)=-\tr\left[\left((V^*A(X)V)^{-\frac{1}{2}}V^*B(Y)^{-1}V(V^*A(X)V)^{-\frac{1}{2}}\right)^{-s}\right]$. Since $-1\leq p,q\leq 0$, both $A^{-1}$ and $B^{-1}$ are operator concave, then $X\mapsto(V^*A(X)V)^{-1}$ is operator concave, so $\Theta_L$ is jointly concave by \cref{lem:engine I}, therefore \cref{prop:unified lifting} proves the claim.\\
    \\
    Region three: Assume $-1\leq p\leq 0;\:1\leq q\leq2;\:s\geq\frac{1}{p+q}$, the other region is symmetric. Set $M(X,Y):=A(X)^\frac{1}{2}$,  $Z(X,Y):=A(X)^\frac{1}{2}B(Y)A(X)^\frac{1}{2}$ and $\varphi(x_1,\ldots,x_r)=-\sum_j x_j^s$, then\\ $\Theta_L(X,Y)=-\tr\left[\left((V^*\Phi(X^p)^{-1}V)^{-\frac{1}{2}}V^*\Psi(Y^q)V(V^*\Phi(X^p)^{-1}V)^{-\frac{1}{2}}\right)^s\right]$ is jointly concave by \cref{lem:engine II}, therefore \cref{prop:unified lifting} proves the claim.
\end{proof}

\begin{rem}
    At full rank $r=d$, this generalizes the joint concavity / convexity in \cite[Theorem 2.1]{Hiai2016ConcavityII} in the case $f(x)=x^s$.
\end{rem}

\begin{cor}
    Given invertible matrix $M$, let $T(X,Y):=(Y^\frac{q}{2}M^\dagger X^pMY^\frac{q}{2})^s$. For every $X_i,Y_i\in\bP_d$ and every probability distribution $\{\alpha_t:1\leq t\leq n\}$:\\
    (1) If $(p,q,s)$ is in Region one, then for every Horn tuple $(I_1,...,I_n;K)$ in dimension $d$,
    \begin{align}
        \sum_{k\in K}\lambda^\uparrow_k\big(T(\sum_{t=1}^{n}\alpha_tX_t,\sum_{t=1}^{n}\alpha_tY_t)\big)\geq\sum_{t=1}^{n}\alpha_t\sum_{i_t\in I_t}\lambda^\uparrow_{i_t}\big(T(X_t,Y_t)\big).
    \end{align}
    In particular, at rank $r=d$, this recovers the Carlen-Frank-Lieb's joint concavity \cite{CarlenFrankLieb2016SomeOperatorTraceConvexity} of $\tr\:T$.\\
    (2) If $(p,q,s)$ is in Region two or three, then for every Horn tuple $(I_1,...,I_n;K)$ in dimension $d$,
    \begin{align}
        \sum_{k\in K}\lambda^\downarrow_k\big(T(\sum_{t=1}^{n}\alpha_tX_t,\sum_{t=1}^{n}\alpha_tY_t)\big)\leq\sum_{t=1}^{n}\alpha_t\sum_{i_t\in I_t}\lambda^\downarrow_{i_t}\big(T(X_t,Y_t)\big).
    \end{align}
    In particular, at rank $r=d$, this recovers Zhang's joint convexity \cite{Zhang2020WYDCarlenFrankLieb} of $\tr\:T$.
\end{cor}

\section{Geometric means and Multiplicative Horn families}
In this section, we consider geometric means. The scalarizer is, for each $L\in\Gr(r,d)$,
\begin{align}
    \Theta_L(X):=\det(V_L^\dagger X^{-1}V_L)^{-\frac{1}{r}}.
\end{align}
We first prove the Hersch-Zwahlen formula and analytic inputs needed for this family.

\begin{lem}[Multiplicative Hersch-Zwahlen] \label{lemma: multiplicative Hersch-Zwahlen}
    For any $X\in\bP_d$ and $I\subset[d]$,
    \begin{align}
        \big(\prod_{i\in I}\lambda_i^\uparrow(X)\big)^\frac{1}{r}=\max_{L\in\Omega_I(F_\bullet^\uparrow(X))}\Theta_L(X)=\min_{L\in\Omega_{I^\vee}(F_\bullet^\downarrow(X))}\Theta_L(X).
    \end{align}
\end{lem}
\begin{proof}
    Applying \cref{theorem: generalized Hersch-Zwahlen} to the coordinatewise nonincreasing $\varphi(x_1,\ldots,x_r):=(x_1\ldots x_r)^{-\frac{1}{r}}$ we get
    \begin{align}
        \big(\prod_{j\in J}\lambda_j^\uparrow(X)\big)^\frac{1}{r}=\varphi(\lambda_{j_1}^\uparrow(X^{-1}),\ldots,\lambda_{j_r}^\uparrow(X^{-1}))=\max_{L\in\Omega_{J^\vee}(F_\bullet^\downarrow(X^{-1}))}\Theta_{X^{-1},\varphi}(L)=\min_{L\in\Omega_J(F_\bullet^\uparrow(X^{-1}))}\Theta_{X^{-1},\varphi}(L)
    \end{align}
    where $I=J^\vee$, which is exactly the claim since $\Theta_{X^{-1},\varphi}=\Theta_L(X)$.
\end{proof}

\begin{lem} \label{lemma: properties of Theta_L}
    (1) For any $X,Y\in\bP_d$,
    \begin{align}
        \Theta_L(X\#_\alpha Y)\geq\Theta_L(X)^{1-\alpha}\Theta_L(Y)^\alpha.
    \end{align}
    (2) $\Theta_L$ is concave on $\bP_d$.\\
    (3) For every $L\in\Gr(r,d)$,
    \begin{align}
        \Theta_L\big(G_n(w;X_1,...,X_n)\big)\geq\prod_{t=1}^{n}\Theta_L(X_t)^{w_t}.
    \end{align}
\end{lem}
\begin{proof}
    (1) Denote $\Phi_L(X):=V_L^\dagger XV_L$, then $\Theta_L(X)=\det(\Phi_L(X^{-1}))^{-\frac{1}{r}}$, by Ando's inequality,
    \begin{align}
        \Theta_L(X\#_\alpha Y)=\det\Phi_L((X\#_\alpha Y)^{-1})^{-\frac{1}{r}}=\det\Phi_L(X^{-1}\#_\alpha Y^{-1})^{-\frac{1}{r}}&\geq\det[\Phi_L(X^{-1})\#_\alpha\Phi_L(Y^{-1})]^{-\frac{1}{r}}\\
        &=\Theta_L(X)^{1-\alpha}\Theta_L(Y)^\alpha.
    \end{align}
    (2) Since $X\#_\alpha Y\leq(1-\alpha)X+\alpha Y$, by operator-monotonicity of $\Theta_L$, 
    \begin{align}
        \Theta_L((1-\alpha)X+\alpha Y)\geq\Theta_L(X\#_\alpha Y)\geq\Theta_L(X)^{1-\alpha}\Theta_L(Y)^\alpha.
    \end{align}
    Set $a=\Theta_L(X)$, $b=\Theta_L(Y)$, $c=(1-\alpha)a+\alpha b$, $\tilde{a}=\frac{(1-\alpha)a}{c}$ and $\tilde{b}=\frac{\alpha b}{c}$. Since $\Theta_L$ is homogeneous, $\Theta_L(\frac{X}{a})=\Theta_L(\frac{Y}{b})=1$ and
    \begin{align}
        \Theta_L((1-\alpha)X+\alpha Y)=c\Theta_L(\frac{\tilde{a}}{a}X+\frac{\tilde{b}}{b}Y)\geq c\Theta_L(\frac{X}{a})^{\tilde{a}}\Theta_L(\frac{Y}{b})^{\tilde{b}}=c=(1-\alpha)\Theta_L(X)+\alpha\Theta_L(Y).
    \end{align}
    (3) $\log\Theta_L$ is geodesically concave by part (1), since $X\#_\alpha Y$ is the unique geodesic and Karcher mean is the corresponding barycenter, by \cref{thm:jensen for barycenters},
    \begin{align}
        \log\Theta_L(G_n(w;X_1,\ldots,X_n))\geq\sum_tw_t\log\Theta_L(X_t).
    \end{align}
\end{proof}

Now we are ready to prove the Horn inequalities:

\begin{thm} \label{theorem: Horn inequalities for geodesic means}
    For every $X_1,...,X_n\in\bP_d$ and every Horn tuple $(I_1,...,I_n;K)$ of rank $r$,
    \begin{align}
        \bigg(\prod_{k\in K}\lambda_k^\uparrow\big(\sigma_\nu(X_1,...,X_n)\big)\bigg)^\frac{1}{r}\geq\gamma_\nu\bigg(\big(\prod_{i_1\in I_1}\lambda_{i_1}^\uparrow(X_1)\big)^\frac{1}{r},...,\big(\prod_{i_n\in I_n}\lambda_{i_n}^\uparrow(X_n)\big)^\frac{1}{r}\bigg)
    \end{align}
    where
    \begin{align}
        \gamma_\nu(x_1,...,x_n):=\int_{\Delta_n}(\prod_{i=1}^{n}x_i^{w_i})d\nu(w).
    \end{align}
    In particular, at rank $r=d$, this generalizes \cite[Theorem 5.5]{BourinHiai2014JensenMinkowski}.
\end{thm}
\begin{proof}
    Set $\Phi:=\sigma_\nu$, $\Gamma:=\gamma_\nu$ and $\Lambda_I(X):=(\prod_{i\in I}\lambda_i^\uparrow(X))^\frac{1}{r}$. Since $\Theta_L$ is concave by \cref{lemma: properties of Theta_L}(2), 
    \begin{align}
        \Theta_L(\Phi(X_1,\ldots,X_n))=\Theta_L(\int_{\Delta_n}G_n(w;X_1,...,X_n) d\nu)\geq\int_{\Delta_n}\Theta_L(G_n(w;X_1,...,X_n)) d\nu.
    \end{align}
    By \cref{lemma: properties of Theta_L}(3),
    \begin{align}
        \int_{\Delta_n}\Theta_L(G_n(w;X_1,...,X_n)) d\nu\geq\int_{\Delta_n}\prod_{t=1}^{n}\Theta_L(X_t)^{w_t} d\nu=\gamma_\nu(\Theta_L(X_1),\ldots,\Theta_L(X_n)),
    \end{align}
    so (L2) of \cref{thm: lifting principle} is satisfied; and \cref{lemma: multiplicative Hersch-Zwahlen} shows (L1), therefore \cref{thm: lifting principle} applies and
    \begin{align}
        \bigg(\prod_{k\in K}\lambda_k^\uparrow\big(\sigma_\nu(X_1,...,X_n)\big)\bigg)^\frac{1}{r}=\Lambda_K(\Phi(X_1,\ldots,X_n))&\geq\Gamma(\Lambda_{I_1}(X_1),\ldots,\Lambda_{I_n}(X_n))\\
        &=\gamma_\nu\bigg(\big(\prod_{i_1\in I_1}\lambda_{i_1}^\uparrow(X_1)\big)^\frac{1}{r},\dots,\big(\prod_{i_n\in I_n}\lambda_{i_n}^\uparrow(X_n)\big)^\frac{1}{r}\bigg).
    \end{align}
\end{proof}

\begin{cor}
    Taking $\nu=\delta_w$, we get 
    \begin{align}
        \prod_{k\in K}\lambda_k^\uparrow\big(G_n(w;X_1,...,X_n)\big)\geq\prod_{t=1}^{n}\big(\prod_{i_t\in I_t}\lambda_{i_t}^\uparrow(X_t)\big)^{w_t}.
    \end{align}
    Equivalently, for every $X_1,...,X_n\in\bP_d$, there exist unitaries $U_1,...,U_n$ such that
    \begin{align}
        \log\big(G_n(w;X_1,...,X_n)\big)\geq\sum_tw_tU_t\log(X_t)U_t^\dagger.
    \end{align}
    Since $\log\det(G_n)=\sum_t w_t\log\det(X_t)$, the equality holds. In particular, for every $X,Y\in\bP_d$, there exist unitaries $U,V$ such that
    \begin{align}
        \log(X\#_\alpha Y)=(1-\alpha)U\log(X)U^\dagger+\alpha V\log(Y)V^\dagger.
    \end{align}
\end{cor}

\section{Conclusion}

In this article, we developed a two-stage framework for lifting majorization inequalities to complete Horn inequalities. The geometric stage is provided by the generalized Hersch--Zwahlen variational formula in \cref{theorem: generalized Hersch-Zwahlen}, which represents coordinatewise monotone functions of arbitrary selected eigenvalue sublists as extrema of scalarizers over Schubert-constrained compressions. Combining this formula with the intersection property of Horn tuples yields the lifting principle in \cref{prop:unified lifting}. The remaining step is family-specific and analytic: one verifies the appropriate concavity, convexity, or mean inequality for the compressed scalarizer. This separation isolates the common geometric mechanism behind the lifting and provides a reusable route from majorization inequalities to the complete Horn system and, in the additive setting, to unitary-orbit inequalities.

We then applied this framework to the three families motivating the article. The convex-function result in \cref{theorem: Aujla-Silva Horn} upgrades the Aujla--Silva weak-majorization inequality to complete Horn and unitary-orbit Jensen inequalities. The positive-power result in \cref{theorem: Two-variable positive-map powers} gives selected-eigenvalue and unitary-orbit extensions of the Carlen--Frank--Lieb--Zhang concavity and convexity theory. The multiplicative result in \cref{theorem: Horn inequalities for geodesic means} establishes complete Horn inequalities for multivariable geodesic means and, for Karcher means, yields unitary-orbit inequalities after taking logarithms. At full rank, these results recover the corresponding trace- or determinant-level statements, while the intermediate-rank inequalities retain spectral information that is invisible to scalarizations.

For quantum information theory, the principal gain is that the additive Horn extensions are equivalent to matrix-valued comparisons on unitary orbits. Such inequalities can be passed through positive maps or quantum channels and tested against observables before taking traces, entropies, or divergences. They may therefore serve as operator-level antecedents or refinements of scalar quantum-information inequalities. More broadly, the framework suggests a systematic route to further examples: identify a scalarizer adapted to the spectral quantity of interest, establish its behavior under compression, and then invoke the lifting principle to obtain complete Horn and, when applicable, unitary-orbit inequalities.

\section*{Acknowledgement and Disclosure of AI usage}
\paragraph*{Acknowledgement} We acknowledge helpful discussions with Mohammad Alhejji, Shiliang Gao, Felix Leditzky and Haonan Zhang. This work was supported by a grant through the IBM-Illinois Discovery Accelerator Institute as well as National Science Foundation Grant No. 2426103.

\paragraph*{Disclosure of AI usage}
The main idea and overall structure of this article were established in early 2026. The literature search and the refinement and polishing of the proofs have been assisted by GPT-5.2 and subsequent models. The author has reviewed and verified the entire article and assumes full responsibility for its accuracy and correctness.

\printbibliography

\end{document}